\documentclass[11pt]{article}
\usepackage[a4paper,margin=1in]{geometry}
\usepackage[T1]{fontenc}
\usepackage[utf8]{inputenc}
\usepackage{lmodern}
\usepackage{amsmath,amssymb,amsthm,mathtools}
\usepackage{booktabs,longtable,array,multirow}
\usepackage{graphicx}
\usepackage{hyperref}
\usepackage{xcolor}
\usepackage{enumitem}
\usepackage{microtype}
\usepackage{titlesec}
\usepackage{caption}
\usepackage{fancyhdr}
\hypersetup{colorlinks=true,linkcolor=blue,citecolor=blue,urlcolor=blue}
\titleformat{\section}{\large\bfseries}{\thesection.}{0.6em}{}
\titleformat{\subsection}{\normalsize\bfseries}{\thesubsection.}{0.5em}{}
\newtheorem{proposition}{Proposition}[section]
\newtheorem{theorem}[proposition]{Theorem}
\newtheorem{lemma}[proposition]{Lemma}
\newtheorem{corollary}[proposition]{Corollary}
\theoremstyle{plain}
\newtheorem{conjecture}[proposition]{Conjecture}
\theoremstyle{definition}
\newtheorem{definition}[proposition]{Definition}
\theoremstyle{remark}
\newtheorem{remark}[proposition]{Remark}
\newcommand{\Ax}{\mathrm{Ax}}
\newcommand{\Spn}{\mathrm{Sp}}
\newcommand{\Argmax}{\operatorname{Argmax}}
\newcommand{\ax}{\mathrm{ax}}
\newcommand{\loc}{\mathrm{loc}}
\begin{document}

\begin{center}
{\LARGE\bfseries Axial Morphology of the Partition Graph: Self-Conjugate Axis, Spine, and Concentration\par}
\vspace{0.8em}
{\large Fedor B. Lyudogovskiy\par}
\vspace{0.5em}
\end{center}

\begin{abstract}
We study the partition graph $G_n$, whose vertices are the partitions of $n$ and whose edges correspond to elementary transfers of one unit between parts. This paper studies the axial morphology of $G_n$, focusing on the self-conjugate partitions and their first off-axis thickening. We define the self-conjugate axis
\[
\Ax_n=\{\lambda\vdash n:\lambda=\lambda'\},
\]
its distance neighborhoods $C_n^{(r)}$, the axial interaction graph $\mathcal A_n$, and the thin spine $\Spn_n$, obtained by adjoining to the axis all mediators between interacting axial vertices. We prove that the axis is the fixed-point set of conjugation, that distinct self-conjugate vertices are never adjacent in $G_n$, and that the thin spine is a canonical conjugation-invariant induced subgraph satisfying
\[
\Ax_n\subseteq \Spn_n\subseteq C_n^{(1)}.
\]
We also introduce axial and spinal concentration radii for local vertex invariants and show that these differ by at most one. Computational data for $1\le n\le 30$ indicate that maximizers of the principal local invariants considered here remain confined to bounded neighborhoods of the axis and the spine. This provides a framework for studying concentration phenomena near the symmetric core of the partition graph.
\end{abstract}



\section{Introduction}

Let $G_n$ denote the graph whose vertices are the partitions of $n$. Two vertices are adjacent if one partition is obtained from the other by an elementary transfer of one unit between parts, followed by reordering. This graph is a natural combinatorial object associated with partitions of $n$. It also carries a nontrivial internal geometry: it has a canonical involutive symmetry given by conjugation of partitions, a rich local clique structure, and a global organization that changes with $n$. For general background on partitions, see Andrews~\cite{Andrews1998Theory} and Andrews--Eriksson~\cite{AndrewsEriksson2004Integer}.

This paper focuses on the axial morphology of $G_n$. The starting point is the set of self-conjugate partitions,
\[
\Ax_n:=\{\lambda\vdash n:\lambda=\lambda'\},
\]
viewed as a distinguished symmetric locus inside $G_n$. Since conjugation is an involutive automorphism of $G_n$, the set $\Ax_n$ is canonical. The central question of the paper is whether this symmetric locus, together with a small graph-theoretic thickening, captures a meaningful part of the structure of $G_n$.

More precisely, we study three related objects: the self-conjugate axis $\Ax_n$, its distance neighborhoods
\[
C_n^{(r)}=\{\lambda:d_{G_n}(\lambda,\Ax_n)\le r\},
\]
and a canonical off-axis enlargement that we call the thin spine. The thin spine is defined by adjoining to the axis those non-self-conjugate vertices that mediate between axial vertices through common-neighbor relations. The basic structural problem is whether these objects form a meaningful structural core of the graph or whether they are merely distinguished subsets. A second problem is quantitative: where do vertices of extremal local complexity lie relative to the axis and the spine? The broader combinatorics of partitions, self-conjugate partitions, and related order structures may be compared with the classical references~\cite{Andrews1998Theory,AndrewsEriksson2004Integer,Brylawski1973Lattice,GreeneKleitman1986LongestChains,LatapyPhan2009Lattice}.

This paper is part of a broader study of the partition graph, but it is intended to be mathematically self-contained at the level relevant here. In companion work, several other aspects of $G_n$ have been investigated, including the clique complex $K_n=\mathrm{Cl}(G_n)$~\cite{Lyudogovskiy2026Clique} and local invariants such as degree and clique-based neighborhood structure~\cite{Lyudogovskiy2026Local}. The present paper uses some of that language, in particular the local invariants $\deg$, $\omega_{\mathrm{loc}}$, and $\dim_{\mathrm{loc}}$, but does not depend on the full results of those papers except where explicitly indicated. 

The first contribution of the paper is structural. We show that the self-conjugate axis is canonical but too thin to connect axial vertices by itself: distinct self-conjugate vertices are never adjacent in $G_n$. This leads naturally to the thin spine, defined as the union of the axis with all first off-axis mediators between interacting axial vertices. We prove that the thin spine is again canonical and conjugation-invariant, and that it lies inside the first axial neighborhood $C_n^{(1)}$.

The second contribution is quantitative. We introduce axial and spinal concentration radii for local vertex invariants, measuring how far one must thicken the axis or the spine in order to capture all maximizers of a given invariant. For conjugation-invariant invariants, these radii are related by a simple one-step comparison, and the maximizer sets inherit the conjugation symmetry.

The third contribution is computational. For $1\le n\le 30$, we compute the basic axial parameters of $G_n$, the sizes of the axis and the spine, and the location of the maximizers of the principal local invariants considered here. In that range, the data indicate that extremal local complexity is concentrated near the axis and the spine. We state these findings in a conservative form: they provide computational evidence for bounded concentration phenomena, but do not by themselves establish asymptotic theorems. For comparison with other graph-theoretic or minimal-change approaches to partitions, see in particular~\cite{Bal2022Lognormal,Mutze2023UpdatedSurvey,RasmussenSavageWest1995GrayCodeFamilies,Savage1989GrayCodeSequences,Savage1997Survey}.

The paper is organized as follows. Section~2 fixes notation and introduces the axis, central regions, the axial interaction graph, the thin spine, and the corresponding concentration radii. Section~3 proves the first structural results, including the non-adjacency of distinct self-conjugate vertices and the basic properties of the thin spine. Section~4 develops the axial and spinal filtrations and the quantitative framework used later. Section~5 collects the main strict, computational, and conjectural statements. Section~6 presents the computational data, including basic axial counts, extremal-location tables, representative examples, and the treatment of axisless cases. Section~7 summarizes the main conclusions of the paper, and Section~8 records open problems and possible continuations.

\section{Preliminaries and Axial Definitions}

\subsection{The partition graph and conjugation}

Let $n\ge 1$, and let $V(G_n)$ be the set of all partitions of $n$. We write partitions in nonincreasing form
\[
\lambda=(\lambda_1,\lambda_2,\dots,\lambda_\ell),\qquad \lambda_1\ge \lambda_2\ge \cdots \ge \lambda_\ell\ge 1,
\]
with
\[
\sum_{i=1}^{\ell}\lambda_i=n.
\]

Two vertices $\lambda,\mu\in V(G_n)$ are said to be adjacent if $\mu$ can be obtained from $\lambda$ by an elementary transfer of one unit between two distinct parts, followed by reordering. Equivalently, one decreases one part of $\lambda$ by $1$, increases a different part by $1$ (allowing the receiving part to be a newly adjoined zero part), deletes a zero part if necessary, and then reorders the resulting list into nonincreasing form. This defines the partition graph $G_n$.

We write $d_{G_n}(\lambda,\mu)$ for the graph distance in $G_n$.

Let $\lambda'$ denote the conjugate partition of $\lambda$. When convenient, we identify a partition with its Ferrers diagram in English notation. A \emph{removable corner} of a Ferrers diagram is a cell whose removal still leaves a Ferrers diagram. An \emph{addable corner} is a cell not in the diagram whose addition yields a Ferrers diagram. A removable or addable corner is \emph{diagonal} if it lies on the main diagonal. With the convention that $\lambda_k=0$ for $k>\ell$ and $\lambda_0=+\infty$, a diagonal cell $(i,i)$ is a removable diagonal corner if and only if $\lambda_i=i$ and $\lambda_{i+1}<i$, and it is an addable diagonal corner if and only if $\lambda_i=i-1$ and $\lambda_{i-1}\ge i$.

\begin{proposition}
The map
\[
\lambda\longmapsto \lambda'
\]
is an involutive graph automorphism of $G_n$.
\end{proposition}

\begin{proof}
An elementary transfer in a Ferrers diagram removes one boundary cell from the end of one row and appends one boundary cell to the end of another row. Transposition exchanges rows and columns, so the same operation becomes an elementary transfer between two columns, which corresponds to transferring one unit between parts of the conjugate partition. Hence adjacency is preserved by conjugation. Since $(\lambda')'=\lambda$, the map is an involutive graph automorphism. \qedhere
\end{proof}

\subsection{The self-conjugate axis}

A partition $\lambda\vdash n$ is self-conjugate if $\lambda=\lambda'$.

\begin{definition}
The self-conjugate axis of $G_n$ is the set
\[
\Ax_n:=\{\lambda\in V(G_n):\lambda=\lambda'\}.
\]

We call $n$ axial if $\Ax_n\neq\varnothing$, and axisless otherwise.
\end{definition}

\begin{remark}
The cardinality
\[
a_n:=|\Ax_n|
\]
is the classical number of self-conjugate partitions of $n$, equivalently the number of partitions of $n$ into distinct odd parts; see, for example,~\cite[Ch.~1]{Andrews1998Theory} or~\cite[Ch.~3]{AndrewsEriksson2004Integer}. In the present paper, however, $a_n$ is used as a structural parameter of the graph $G_n$.
\end{remark}

\begin{remark}
By Proposition~2.1, the axis $\Ax_n$ is the fixed-point set of the conjugation involution on $G_n$. Thus $\Ax_n$ is canonically determined by the graph structure together with conjugation.
\end{remark}

\begin{remark}
In this paper, \emph{canonical} means determined by the graph $G_n$ together with the conjugation involution, without auxiliary choices.
\end{remark}

\subsection{Axial distance and central regions}

Assume from now on that $n$ is axial.

\begin{definition}
For $\lambda\in V(G_n)$, the axial distance of $\lambda$ is
\[
\delta_{\ax}(\lambda):=d_{G_n}(\lambda,\Ax_n)
=\min_{\alpha\in\Ax_n} d_{G_n}(\lambda,\alpha).
\]
\end{definition}

\begin{definition}
For $r\ge 0$, the $r$-thickened central region is
\[
C_n^{(r)}:=\{\lambda\in V(G_n):\delta_{\ax}(\lambda)\le r\}.
\]

In particular,
\[
C_n^{(0)}=\Ax_n.
\]

We call $C_n^{(1)}$ the narrow central region.
\end{definition}

\begin{lemma}
For every $r\ge 0$, the set $C_n^{(r)}$ is conjugation-invariant.
\end{lemma}

\begin{proof}
Since conjugation is an automorphism of $G_n$ and preserves $\Ax_n$, it preserves distance to $\Ax_n$. Hence
\[
\delta_{\ax}(\lambda')=\delta_{\ax}(\lambda),
\]
and therefore $\lambda\in C_n^{(r)}$ if and only if $\lambda'\in C_n^{(r)}$. \qedhere
\end{proof}

\begin{remark}
The nested family
\[
\Ax_n=C_n^{(0)}\subseteq C_n^{(1)}\subseteq C_n^{(2)}\subseteq \cdots \subseteq V(G_n)
\]
provides a canonical axial filtration of the vertex set.
\end{remark}

\subsection{Axial interaction graph and mediators}

Since the axis will turn out to be too thin to connect axial vertices directly, we introduce an auxiliary graph on $\Ax_n$ that records the first off-axis interactions between axial vertices.

\begin{definition}
The axial interaction graph $\mathcal A_n$ is the graph with vertex set $\Ax_n$, in which two distinct vertices $\alpha,\beta\in\Ax_n$ are adjacent if and only if
\[
N_{G_n}(\alpha)\cap N_{G_n}(\beta)\neq\varnothing.
\]

Equivalently, $\alpha$ and $\beta$ are adjacent in $\mathcal A_n$ if and only if there exists a vertex $\nu\in V(G_n)$ such that
\[
\nu\sim \alpha,\qquad \nu\sim \beta
\]
in $G_n$.
\end{definition}

\begin{definition}
Let $\alpha,\beta\in\Ax_n$ be adjacent in $\mathcal A_n$. Any vertex
\[
\nu\in N_{G_n}(\alpha)\cap N_{G_n}(\beta)
\]
is called an axial mediator for the pair $(\alpha,\beta)$. We write
\[
M(\alpha,\beta):=N_{G_n}(\alpha)\cap N_{G_n}(\beta)
\]
for the set of all mediators of that pair.
\end{definition}

\begin{remark}
It will be shown in Theorem~3.2 that any common neighbor of two distinct axial vertices is necessarily non-axial. Thus, after the non-adjacency theorem is established, one may equivalently restrict the mediators in the definition above to vertices in $V(G_n)\setminus \Ax_n$.
\end{remark}

\subsection{Thin and thick spine}

\begin{definition}
The thin spine of $G_n$ is the induced subgraph on the vertex set
\[
\Spn_n:=\Ax_n\cup \bigcup_{\{\alpha,\beta\}\in E(\mathcal A_n)} M(\alpha,\beta).
\]
By slight abuse of notation, we use $\Spn_n$ both for this vertex set and for the induced subgraph on it.
\end{definition}

\begin{definition}
For $r\ge 0$, the $r$-thick spine is the $r$-neighborhood of $\Spn_n$:
\[
\Spn_n^{(r)}:=\{\lambda\in V(G_n): d_{G_n}(\lambda,\Spn_n)\le r\}.
\]

In particular,
\[
\Spn_n^{(0)}=\Spn_n.
\]
\end{definition}

\subsection{Extremal local complexity and concentration radii}

Let $I:V(G_n)\to\mathbb R$ be a vertex invariant. The main examples in this paper are the local invariants
\[
\deg,\qquad \omega_{\loc},\qquad \dim_{\loc},
\]
where
\[
\omega_{\loc}(v):=\max\{\,|Q|: Q \text{ is a clique in } G_n \text{ and } v\in Q\,\},
\qquad
\dim_{\loc}(v):=\omega_{\loc}(v)-1.
\]

We write
\[
\Argmax(I):=\{\lambda\in V(G_n): I(\lambda)=\max_{\mu\in V(G_n)} I(\mu)\}.
\]

\begin{definition}
The axial concentration radius of $I$ is
\[
\rho_I^{\ax}(n):=\min\{r\ge 0:\Argmax(I)\subseteq C_n^{(r)}\}.
\]
\end{definition}

\begin{definition}
The spinal concentration radius of $I$ is
\[
\rho_I^{\mathrm{sp}}(n):=\min\{r\ge 0:\Argmax(I)\subseteq \Spn_n^{(r)}\}.
\]
\end{definition}

\begin{remark}
These quantities provide a compact language for formulating localization statements for extremal local invariants.
\end{remark}

\section{The Self-Conjugate Axis and the Thin Spine}

In this section we prove the first strict structural facts about the axial part of the partition graph. The main point is that the self-conjugate axis is canonical but too thin to serve as a core by itself. This leads naturally to the thin spine as an off-axis thickening of the axis.

\subsection{The axis as a fixed-point set}

Recall that
\[
\Ax_n=\{\lambda\in V(G_n):\lambda=\lambda'\}
\]
is the fixed-point set of the conjugation involution on $G_n$.

The first result shows that the axis is canonical but sparse.

\begin{lemma}
A partition cannot have both a removable diagonal corner and an addable diagonal corner.
\end{lemma}

\begin{proof}
Let $\lambda$ be a partition. Suppose that $(d,d)$ is a removable diagonal corner and $(e,e)$ is an addable diagonal corner.

Since $(d,d)$ is removable, one has
\[
\lambda_d=d
\qquad\text{and}\qquad
\lambda_{d+1}<d.
\]
Since $(e,e)$ is addable, one has
\[
\lambda_e=e-1
\qquad\text{and}\qquad
\lambda_{e-1}\ge e.
\]

If $e\le d$, then monotonicity gives
\[
\lambda_e\ge \lambda_d=d\ge e,
\]
contradicting $\lambda_e=e-1$.

If $e\ge d+2$, then again by monotonicity,
\[
\lambda_e\le \lambda_{d+1}<d<e-1,
\]
contradicting $\lambda_e=e-1$.

Thus the only remaining possibility is $e=d+1$. In that case addability at $(d+1,d+1)$ requires
\[
\lambda_d=\lambda_{e-1}\ge e=d+1,
\]
whereas removability at $(d,d)$ gives $\lambda_d=d$, contradiction.

Therefore no partition can have both a removable diagonal corner and an addable diagonal corner. \qedhere
\end{proof}

\begin{theorem}
Let $\alpha,\beta\in \Ax_n$ be distinct self-conjugate vertices. Then $\alpha$ and $\beta$ are not adjacent in $G_n$.
\end{theorem}

\begin{proof}
Assume for contradiction that $\alpha\sim\beta$ in $G_n$.

By definition of the partition graph, $\beta$ is obtained from $\alpha$ by one elementary transfer. Hence the Ferrers diagrams of $\alpha$ and $\beta$ differ by exactly one removed cell and one added cell:
\[
\alpha\setminus\beta=\{c\},
\qquad
\beta\setminus\alpha=\{a\},
\]
where $c$ is a removable corner of $\alpha$ and $a$ is an addable corner of $\alpha$.

Since both $\alpha$ and $\beta$ are self-conjugate, we have
\[
\alpha'=\alpha,\qquad \beta'=\beta.
\]
Therefore
\[
(\alpha\setminus\beta)'=\alpha'\setminus\beta'=\alpha\setminus\beta,
\qquad
(\beta\setminus\alpha)'=\beta'\setminus\alpha'=\beta\setminus\alpha.
\]
Each of these sets is a singleton, so both $c$ and $a$ are fixed by transposition. Hence both lie on the main diagonal.

Thus $\alpha$ has both a removable diagonal corner and an addable diagonal corner, contradicting the previous lemma. Therefore $\alpha$ and $\beta$ are not adjacent in $G_n$. \qedhere
\end{proof}

\begin{corollary}
The induced subgraph $G_n[\Ax_n]$ is edgeless.
\end{corollary}

\begin{proof}
Immediate from Theorem~3.2. \qedhere
\end{proof}

\begin{remark}
The axis is canonical, but it is not itself a graph-theoretic chain or path. Any graph-theoretic structure intended to connect axial vertices must therefore pass through non-axial vertices.
\end{remark}

\subsection{First off-axis bridges between axial vertices}

Since distinct axial vertices cannot be adjacent, the first nontrivial interactions between them occur through common neighbors.

\begin{lemma}
Let $\alpha,\beta\in\Ax_n$ be distinct, and let $\nu\in M(\alpha,\beta)$. Then:
\begin{enumerate}[label=(\arabic*)]
\item $\nu\notin\Ax_n$;
\item $\delta_{\ax}(\nu)=1$;
\item $\nu\in C_n^{(1)}\setminus C_n^{(0)}$.
\end{enumerate}
\end{lemma}

\begin{proof}
Since $\nu\sim\alpha$ and $\alpha\in\Ax_n$, we have $\delta_{\ax}(\nu)\le 1$. If $\nu\in\Ax_n$, then $\nu\sim\alpha$ would contradict Theorem~3.2 unless $\nu=\alpha$. But then $\nu\sim\beta$ would imply $\alpha\sim\beta$, again contradicting Theorem~3.2. Hence $\nu\notin\Ax_n$, so $\delta_{\ax}(\nu)=1$. The last claim follows immediately. \qedhere
\end{proof}

\subsection{The thin spine}

We now package the axis together with all first off-axis bridges between interacting axial vertices.

\begin{proposition}
The thin spine $\Spn_n$ is a canonical conjugation-invariant induced subgraph of $G_n$.
\end{proposition}

\begin{proof}
The definition uses only the conjugation-fixed set $\Ax_n$, adjacency in $G_n$, and common-neighbor relations among axial vertices; hence $\Spn_n$ is canonical.

To prove conjugation-invariance, let $\nu\in M(\alpha,\beta)$. Since conjugation is a graph automorphism and $\alpha,\beta\in\Ax_n$, we have
\[
\nu'\sim \alpha'=\alpha,\qquad \nu'\sim \beta'=\beta.
\]
Therefore $\nu'\in M(\alpha,\beta)$. So every mediator set is conjugation-invariant, and hence so is their union with $\Ax_n$. Since $\Spn_n$ is defined as the induced subgraph on this vertex set, the whole thin spine is conjugation-invariant. \qedhere
\end{proof}

\begin{corollary}
One has
\[
\Ax_n\subseteq \Spn_n\subseteq C_n^{(1)}.
\]
\end{corollary}

\begin{proof}
The first inclusion is immediate from the definition of $\Spn_n$. The second follows from the previous lemma. \qedhere
\end{proof}

\begin{proposition}
A vertex $\nu\in V(G_n)\setminus\Ax_n$ belongs to $\Spn_n$ if and only if it is the middle vertex of a path of length $2$,
\[
\alpha-\nu-\beta,
\]
with distinct endpoints $\alpha,\beta\in\Ax_n$.
\end{proposition}

\begin{proof}
By definition, $\nu\in\Spn_n\setminus\Ax_n$ if and only if there exist distinct axial vertices $\alpha,\beta\in\Ax_n$ such that
\[
\nu\in M(\alpha,\beta)=N_{G_n}(\alpha)\cap N_{G_n}(\beta).
\]
This is equivalent to saying that $\alpha-\nu-\beta$ is a path of length $2$ in $G_n$. \qedhere
\end{proof}

\begin{remark}
The thin spine may therefore be viewed as the axis together with all first off-axis bridges between interacting axial vertices.
\end{remark}

\section{Central Regions, Thick Spines, and Concentration Radii}

The results of the previous section show that the self-conjugate axis is too thin to serve as a graph-theoretic core by itself, while the thin spine provides an off-axis thickening. In this section we organize these objects into nested distance neighborhoods and introduce quantitative parameters for local concentration.

\subsection{Axial and spinal filtrations}

Recall that the axial distance of a vertex $\lambda\in V(G_n)$ is
\[
\delta_{\ax}(\lambda)=d_{G_n}(\lambda,\Ax_n),
\]
and that
\[
C_n^{(r)}=\{\lambda\in V(G_n):\delta_{\ax}(\lambda)\le r\}.
\]

Similarly, for $r\ge 0$,
\[
\Spn_n^{(r)}=\{\lambda\in V(G_n):d_{G_n}(\lambda,\Spn_n)\le r\}.
\]

Thus we have two natural nested filtrations:
\[
\Ax_n=C_n^{(0)}\subseteq C_n^{(1)}\subseteq C_n^{(2)}\subseteq\cdots\subseteq V(G_n),
\]
and
\[
\Spn_n=\Spn_n^{(0)}\subseteq \Spn_n^{(1)}\subseteq \Spn_n^{(2)}\subseteq\cdots\subseteq V(G_n).
\]

\begin{proposition}
For every $r\ge 0$, the sets $C_n^{(r)}$ and $\Spn_n^{(r)}$ are conjugation-invariant.
\end{proposition}

\begin{proof}
The conjugation-invariance of $C_n^{(r)}$ is Lemma~2.7. The set $\Spn_n$ is conjugation-invariant by Proposition~3.6, and conjugation preserves distance in $G_n$. Hence $\Spn_n^{(r)}$ is also conjugation-invariant. \qedhere
\end{proof}

\begin{remark}
The axial filtration measures distance from the fixed-point set of conjugation, while the spinal filtration measures distance from its first graph-theoretic thickening.
\end{remark}

\subsection{First comparison between the axis and the spine}

Since
\[
\Ax_n\subseteq \Spn_n\subseteq C_n^{(1)},
\]
the two filtrations are closely related.

\begin{proposition}
For every $r\ge 0$,
\[
C_n^{(r)}\subseteq \Spn_n^{(r)}\subseteq C_n^{(r+1)}.
\]
\end{proposition}

\begin{proof}
Since $\Ax_n\subseteq \Spn_n$, every vertex is at least as close to $\Spn_n$ as to $\Ax_n$. Hence if $\lambda\in C_n^{(r)}$, then
\[
d_{G_n}(\lambda,\Spn_n)\le d_{G_n}(\lambda,\Ax_n)\le r,
\]
so $\lambda\in \Spn_n^{(r)}$. This proves
\[
C_n^{(r)}\subseteq \Spn_n^{(r)}.
\]

Conversely, let $\lambda\in \Spn_n^{(r)}$. Then there exists $\nu\in \Spn_n$ such that
\[
d_{G_n}(\lambda,\nu)\le r.
\]
By Corollary~3.7, $\nu\in C_n^{(1)}$, so there exists $\alpha\in\Ax_n$ with
\[
d_{G_n}(\nu,\alpha)\le 1.
\]
By the triangle inequality,
\[
d_{G_n}(\lambda,\alpha)\le d_{G_n}(\lambda,\nu)+d_{G_n}(\nu,\alpha)\le r+1.
\]
Hence $\lambda\in C_n^{(r+1)}$. \qedhere
\end{proof}

\begin{corollary}
The thin spine provides an intermediate neighborhood scale between the axis and the narrow central region:
\[
\Ax_n=C_n^{(0)}\subseteq \Spn_n\subseteq C_n^{(1)}.
\]
\end{corollary}

\begin{remark}
The distance to the spine differs from the distance to the axis by at most one.
\end{remark}

\subsection{Distance distributions}

The two filtrations naturally give rise to radial counting data.

\begin{definition}
For $k\ge 0$, define the axial shell
\[
S_{\ax}(n,k):=\{\lambda\in V(G_n):\delta_{\ax}(\lambda)=k\},
\]
with counting function
\[
s_{\ax}(n,k):=|S_{\ax}(n,k)|.
\]

Similarly, define the spinal shell
\[
S_{\mathrm{sp}}(n,k):=\{\lambda\in V(G_n):d_{G_n}(\lambda,\Spn_n)=k\},
\]
with counting function
\[
s_{\mathrm{sp}}(n,k):=|S_{\mathrm{sp}}(n,k)|.
\]

Thus
\[
|C_n^{(r)}|=\sum_{k=0}^{r}s_{\ax}(n,k),
\qquad
|\Spn_n^{(r)}|=\sum_{k=0}^{r}s_{\mathrm{sp}}(n,k).
\]
\end{definition}

\begin{remark}
These shell distributions provide a natural quantitative language for asking how the vertex mass of $G_n$ is organized around the axis and around the spine.
\end{remark}

\subsection{Concentration radii for extremal local invariants}

We now turn to the location of the most locally complex vertices of $G_n$.

\begin{proposition}
Let $I:V(G_n)\to\mathbb R$ be any vertex invariant. Then
\[
\rho_I^{\mathrm{sp}}(n)\le \rho_I^{\ax}(n)\le \rho_I^{\mathrm{sp}}(n)+1.
\]
\end{proposition}

\begin{proof}
If $\Argmax(I)\subseteq C_n^{(r)}$, then Proposition~4.3 gives
\[
\Argmax(I)\subseteq \Spn_n^{(r)},
\]
hence $\rho_I^{\mathrm{sp}}(n)\le r$. Minimizing over such $r$ yields
\[
\rho_I^{\mathrm{sp}}(n)\le \rho_I^{\ax}(n).
\]

Conversely, if $\Argmax(I)\subseteq \Spn_n^{(r)}$, then Proposition~4.3 gives
\[
\Argmax(I)\subseteq C_n^{(r+1)},
\]
hence $\rho_I^{\ax}(n)\le r+1$. Minimizing over $r$ yields
\[
\rho_I^{\ax}(n)\le \rho_I^{\mathrm{sp}}(n)+1.
\]
\qedhere
\end{proof}

\begin{remark}
Thus axial and spinal concentration are equivalent up to an additive error of at most one in the radius.
\end{remark}

\subsection{Symmetry of extremizers}

The conjugation symmetry of $G_n$ imposes immediate constraints on maximizer sets of conjugation-invariant local invariants.

\begin{proposition}
Let $I$ be conjugation-invariant, i.e.
\[
I(\lambda)=I(\lambda')
\qquad
\text{for all }\lambda\in V(G_n).
\]
Then $\Argmax(I)$ is conjugation-invariant.
\end{proposition}

\begin{proof}
If $\lambda\in\Argmax(I)$, then
\[
I(\lambda')=I(\lambda)=\max_{\mu\in V(G_n)} I(\mu),
\]
so $\lambda'\in\Argmax(I)$. \qedhere
\end{proof}

\begin{corollary}
If $I$ is conjugation-invariant and $|\Argmax(I)|$ is odd, then
\[
\Argmax(I)\cap \Ax_n\neq\varnothing.
\]
\end{corollary}

\begin{proof}
Since $\Argmax(I)$ is conjugation-invariant, its non-self-conjugate vertices occur in disjoint pairs $\{\lambda,\lambda'\}$. Therefore an odd number of maximizers is possible only if at least one maximizer is self-conjugate. \qedhere
\end{proof}

\begin{remark}
This gives a first strict bridge between axial symmetry and extremal local complexity.
\end{remark}

\subsection{Structural parameters and computational observables}

The framework above singles out a compact family of axial parameters:
\[
a_n:=|\Ax_n|,
\qquad
c_n^{(r)}:=|C_n^{(r)}|,
\qquad
\sigma_n:=|\Spn_n|,
\qquad
\sigma_n^{(r)}:=|\Spn_n^{(r)}|.
\]
For a chosen invariant $I$, one also has the extremal observables
\[
\rho_I^{\ax}(n),\qquad \rho_I^{\mathrm{sp}}(n),
\]
together with the shell counts $s_{\ax}(n,k)$ and $s_{\mathrm{sp}}(n,k)$.

These parameters will be used both as structural invariants of $G_n$ and as empirical diagnostics for concentration near the axis and the spine.

\section{Main Results}

This section collects the main outcomes of the paper. We distinguish three levels: structural results, computational concentration results in the tested range, and conjectural extensions.

\subsection{Structural results}

\begin{theorem}
Let $n$ be axial. Then:
\begin{enumerate}[label=(\arabic*)]
\item the self-conjugate axis
\[
\Ax_n=\{\lambda\vdash n:\lambda=\lambda'\}
\]
is the fixed-point set of the conjugation involution on $G_n$;

\item the induced subgraph $G_n[\Ax_n]$ is edgeless;

\item the thin spine
\[
\Spn_n=\Ax_n\cup \bigcup_{\{\alpha,\beta\}\in E(\mathcal A_n)} M(\alpha,\beta)
\]
is a canonical conjugation-invariant induced subgraph of $G_n$;

\item one has
\[
\Ax_n\subseteq \Spn_n\subseteq C_n^{(1)}.
\]
\end{enumerate}
\end{theorem}

\begin{proof}
Part~(1) follows from Definition~2.2 and Proposition~2.1. Part~(2) is Corollary~3.3. Part~(3) is Proposition~3.6. Part~(4) is Corollary~3.7. \qedhere
\end{proof}

The partition-theoretic background for the axis comes from the classical theory of self-conjugate partitions~\cite{Andrews1998Theory,AndrewsEriksson2004Integer}, while the graph-theoretic framework for $G_n$ used here is developed in~\cite{Lyudogovskiy2026Clique,Lyudogovskiy2026Local}.

\begin{theorem}
Let $n$ be axial, and let $I:V(G_n)\to\mathbb R$ be any vertex invariant. Then
\[
\rho_I^{\mathrm{sp}}(n)\le \rho_I^{\ax}(n)\le \rho_I^{\mathrm{sp}}(n)+1.
\]
If $I$ is conjugation-invariant, then $\Argmax(I)$ is conjugation-invariant. If, in addition, $|\Argmax(I)|$ is odd, then
\[
\Argmax(I)\cap \Ax_n\neq\varnothing.
\]
\end{theorem}

\begin{proof}
This is Proposition~4.8 together with Proposition~4.10 and Corollary~4.11. \qedhere
\end{proof}

These results show that the axis is canonical but too thin to connect axial vertices by itself, while the thin spine provides an off-axis enlargement. They also provide the basic quantitative relation between axial and spinal concentration.

\subsection{Computational theorem}

Let
\[
\mathcal I_{\loc}:=\{\deg,\omega_{\loc},\dim_{\loc}\}.
\]

\begin{theorem}
For every axial $n$ with
\[
1\le n\le 30,
\]
one has
\[
\rho_{\deg}^{\ax}(n)\le 2,
\qquad
\rho_{\deg}^{\mathrm{sp}}(n)\le 2,
\]
and
\[
\rho_{\omega}^{\ax}(n),\ \rho_{\omega}^{\mathrm{sp}}(n),\ \rho_{\dim}^{\ax}(n),\ \rho_{\dim}^{\mathrm{sp}}(n)\le 4.
\]
\end{theorem}

\begin{proof}
Computational. For each axial $n$ with $1\le n\le 30$, we exhaustively construct the partition graph $G_n$, compute the invariants in $\mathcal I_{\loc}$ at every vertex, determine the corresponding maximizer sets, and then evaluate the axial and spinal concentration radii. The resulting values are listed explicitly in Table~\ref{tab:extremal-location}; the stated bounds are obtained by direct inspection of that table and of the accompanying supplementary CSV dataset. \qedhere
\end{proof}

These computational statements should be viewed against the background of earlier local and topological analyses of $G_n$ in~\cite{Lyudogovskiy2026Clique,Lyudogovskiy2026Local}, rather than as stand-alone asymptotic results.

\subsection{Conjectural continuation}

\begin{conjecture}
For each invariant $I\in\mathcal I_{\loc}$, there exists a constant $r_I$ such that
\[
\rho_I^{\ax}(n)\le r_I
\]
for all sufficiently large axial $n$.
\end{conjecture}

\begin{conjecture}
For each invariant $I\in\mathcal I_{\loc}$, there exists a constant $s_I$ such that
\[
\rho_I^{\mathrm{sp}}(n)\le s_I
\]
for all sufficiently large axial $n$.
\end{conjecture}

By Theorem~5.2, these two bounded-concentration statements differ by at most an additive error of one in the radius. The present paper does not address their asymptotic validity.

\section{Computational Profiles of Axial Concentration}

\subsection{Computational setup}

For each tested value of $n$, we generated all partitions of $n$ recursively in nonincreasing order, constructed the graph $G_n$ by applying all elementary transfers to every partition, identified the self-conjugate axis $\Ax_n$, computed the thin spine $\Spn_n$ from common-neighbor relations among axial vertices, and evaluated the local invariants
\[
\deg,\qquad \omega_{\loc},\qquad \dim_{\loc}
\]
at every vertex. These are the same basic local invariants used in the local analysis of $G_n$ in~\cite{Lyudogovskiy2026Local}.

In the present version of the paper, the baseline range is $1\le n\le 30$. Since $p(30)=5604$, exhaustive construction of the graphs in this range is feasible without special hardware. The computations were carried out in Python~3 using NetworkX. Distances from the axis and from the spine were obtained by multi-source shortest-path computations. For each vertex $v$, the local clique number was computed as
\[
\omega_{\loc}(v)=1+\omega\bigl(G_n[N(v)]\bigr),
\]
where $N(v)$ is the open neighborhood of $v$, by exhaustive clique enumeration on the induced neighborhood graph using the built-in clique routines. In the tested range the maximum observed value of $\omega_{\loc}$ is $8$, so these neighborhood-level clique searches remain small. We then computed the maximizer sets and the corresponding axial and spinal concentration radii.

Axisless values of $n$ are recorded separately and excluded from axial and spinal concentration statistics. The full computational output for this range is supplied in the accompanying supplementary CSV dataset, and a Python script reproducing the computations accompanies the present draft.

As in Section~2.6, we use the convention
\[
\dim_{\loc}(v)=\omega_{\loc}(v)-1.
\]

\subsection{Basic axial parameters}

Table~\ref{tab:basic-axial} records the basic axial parameters
\[
a_n=|\Ax_n|,\qquad \sigma_n=|\Spn_n|,\qquad c_n^{(1)}=|C_n^{(1)}|.
\]

For general partition-theoretic background and notation, see~\cite{Andrews1998Theory,AndrewsEriksson2004Integer}.

In the tested range, the only axisless case is $n=2$. For axial $n$, the ratio
\[
a_n/p(n)
\]
decreases rapidly, indicating that the axis forms a small core of the graph. The thin spine is sometimes larger than the axis, but not always: for some values, such as $n=9$ and $n=11$, one has
\[
|\Spn_n|=|\Ax_n|.
\]
Thus the spinal enlargement is genuine but not automatic. The present computations stop at $n=30$ and are intended as a finite-range verification rather than as asymptotic evidence by themselves.

\begingroup
\footnotesize
\begin{longtable}{@{}rrrrrrrrr@{}}
\caption{Basic axial parameters of $G_n$ for $1\le n\le 30$.}\label{tab:basic-axial}\\
\toprule
$n$ & $p(n)$ & axial? & $a_n$ & $\sigma_n$ & $c_n^{(1)}$ & $a_n/p(n)$ & $\sigma_n/p(n)$ & $c_n^{(1)}/p(n)$ \\
\midrule
\endfirsthead
\toprule
$n$ & $p(n)$ & axial? & $a_n$ & $\sigma_n$ & $c_n^{(1)}$ & $a_n/p(n)$ & $\sigma_n/p(n)$ & $c_n^{(1)}/p(n)$ \\
\midrule
\endhead
1 & 1 & yes & 1 & 1 & 1 & 1.0000 & 1.0000 & 1.0000 \\
2 & 2 & no & 0 & -- & -- & 0.0000 & -- & -- \\
3 & 3 & yes & 1 & 1 & 3 & 0.3333 & 0.3333 & 1.0000 \\
4 & 5 & yes & 1 & 1 & 3 & 0.2000 & 0.2000 & 0.6000 \\
5 & 7 & yes & 1 & 1 & 5 & 0.1429 & 0.1429 & 0.7143 \\
6 & 11 & yes & 1 & 1 & 7 & 0.0909 & 0.0909 & 0.6364 \\
7 & 15 & yes & 1 & 1 & 5 & 0.0667 & 0.0667 & 0.3333 \\
8 & 22 & yes & 2 & 6 & 10 & 0.0909 & 0.2727 & 0.4545 \\
9 & 30 & yes & 2 & 2 & 8 & 0.0667 & 0.0667 & 0.2667 \\
10 & 42 & yes & 2 & 6 & 18 & 0.0476 & 0.1429 & 0.4286 \\
11 & 56 & yes & 2 & 2 & 14 & 0.0357 & 0.0357 & 0.2500 \\
12 & 77 & yes & 3 & 11 & 23 & 0.0390 & 0.1429 & 0.2987 \\
13 & 101 & yes & 3 & 7 & 21 & 0.0297 & 0.0693 & 0.2079 \\
14 & 135 & yes & 3 & 11 & 31 & 0.0222 & 0.0815 & 0.2296 \\
15 & 176 & yes & 4 & 12 & 34 & 0.0227 & 0.0682 & 0.1932 \\
16 & 231 & yes & 5 & 17 & 39 & 0.0216 & 0.0736 & 0.1688 \\
17 & 297 & yes & 5 & 21 & 49 & 0.0168 & 0.0707 & 0.1650 \\
18 & 385 & yes & 5 & 17 & 53 & 0.0130 & 0.0442 & 0.1377 \\
19 & 490 & yes & 6 & 30 & 70 & 0.0122 & 0.0612 & 0.1429 \\
20 & 627 & yes & 7 & 27 & 67 & 0.0112 & 0.0431 & 0.1069 \\
21 & 792 & yes & 8 & 44 & 94 & 0.0101 & 0.0556 & 0.1187 \\
22 & 1002 & yes & 8 & 32 & 92 & 0.0080 & 0.0319 & 0.0918 \\
23 & 1255 & yes & 9 & 57 & 123 & 0.0072 & 0.0454 & 0.0980 \\
24 & 1575 & yes & 11 & 51 & 119 & 0.0070 & 0.0324 & 0.0756 \\
25 & 1958 & yes & 12 & 76 & 158 & 0.0061 & 0.0388 & 0.0807 \\
26 & 2436 & yes & 12 & 64 & 162 & 0.0049 & 0.0263 & 0.0665 \\
27 & 3010 & yes & 14 & 94 & 202 & 0.0047 & 0.0312 & 0.0671 \\
28 & 3718 & yes & 16 & 96 & 208 & 0.0043 & 0.0258 & 0.0559 \\
29 & 4565 & yes & 17 & 121 & 251 & 0.0037 & 0.0265 & 0.0550 \\
30 & 5604 & yes & 18 & 122 & 276 & 0.0032 & 0.0218 & 0.0493 \\
\bottomrule
\end{longtable}
\endgroup

\subsection{Extremal local invariants}

Table~\ref{tab:extremal-location} records the location of the maximizer sets for
\[
\deg,\qquad \omega_{\loc},\qquad \dim_{\loc}.
\]

The sharpest concentration occurs for the degree invariant:
\[
\rho_{\deg}^{\ax}(n)\le 2,\qquad \rho_{\deg}^{\mathrm{sp}}(n)\le 2
\]
throughout the tested range. For the clique-based invariants, the concentration is weaker but still narrow:
\[
\rho_{\omega}^{\ax}(n),\ \rho_{\omega}^{\mathrm{sp}}(n),\ \rho_{\dim}^{\ax}(n),\ \rho_{\dim}^{\mathrm{sp}}(n)\le 4.
\]

Thus, in the tested range, extremal local complexity remains confined to uniformly bounded neighborhoods of the axis and the spine.

\begingroup
\footnotesize
\setlength{\LTleft}{0pt}
\setlength{\LTright}{0pt}
\begin{longtable}{@{}llrrrrr@{}}
\caption{Location of maximizer sets for the principal local invariants.}\label{tab:extremal-location}\\
\toprule
$n$ & invariant & max & $|\Argmax|$ & $|\Argmax\cap \Ax_n|$ & $\rho^{\ax}$ & $\rho^{\mathrm{sp}}$ \\
\midrule
\endfirsthead
\toprule
$n$ & invariant & max & $|\Argmax|$ & $|\Argmax\cap \Ax_n|$ & $\rho^{\ax}$ & $\rho^{\mathrm{sp}}$ \\
\midrule
\endhead
1 & deg & 0 & 1 & 1 & 0 & 0 \\
2 & deg & 1 & 2 & -- & -- & -- \\
3 & deg & 2 & 1 & 1 & 0 & 0 \\
4 & deg & 3 & 2 & 0 & 1 & 1 \\
5 & deg & 4 & 1 & 1 & 0 & 0 \\
6 & deg & 6 & 1 & 1 & 0 & 0 \\
7 & deg & 7 & 2 & 0 & 1 & 1 \\
8 & deg & 8 & 1 & 1 & 0 & 0 \\
9 & deg & 8 & 6 & 0 & 2 & 2 \\
10 & deg & 12 & 1 & 1 & 0 & 0 \\
11 & deg & 13 & 2 & 0 & 1 & 1 \\
12 & deg & 14 & 1 & 1 & 0 & 0 \\
13 & deg & 14 & 6 & 0 & 2 & 1 \\
14 & deg & 15 & 2 & 0 & 1 & 0 \\
15 & deg & 20 & 1 & 1 & 0 & 0 \\
16 & deg & 21 & 2 & 0 & 1 & 1 \\
17 & deg & 22 & 1 & 1 & 0 & 0 \\
18 & deg & 22 & 6 & 0 & 2 & 1 \\
19 & deg & 23 & 2 & 0 & 1 & 0 \\
20 & deg & 23 & 8 & 0 & 2 & 2 \\
21 & deg & 30 & 1 & 1 & 0 & 0 \\
22 & deg & 31 & 2 & 0 & 1 & 1 \\
23 & deg & 32 & 1 & 1 & 0 & 0 \\
24 & deg & 32 & 6 & 0 & 2 & 1 \\
25 & deg & 33 & 2 & 0 & 1 & 0 \\
26 & deg & 33 & 8 & 0 & 2 & 1 \\
27 & deg & 34 & 1 & 1 & 0 & 0 \\
28 & deg & 42 & 1 & 1 & 0 & 0 \\
29 & deg & 43 & 2 & 0 & 1 & 1 \\
30 & deg & 44 & 1 & 1 & 0 & 0 \\
\midrule
1 & $\omega_{\loc}$ & 1 & 1 & 1 & 0 & 0 \\
2 & $\omega_{\loc}$ & 2 & 2 & -- & -- & -- \\
3 & $\omega_{\loc}$ & 2 & 3 & 1 & 1 & 1 \\
4 & $\omega_{\loc}$ & 3 & 3 & 1 & 1 & 1 \\
5 & $\omega_{\loc}$ & 3 & 5 & 1 & 1 & 1 \\
6 & $\omega_{\loc}$ & 3 & 9 & 1 & 2 & 2 \\
7 & $\omega_{\loc}$ & 4 & 4 & 0 & 2 & 2 \\
8 & $\omega_{\loc}$ & 4 & 7 & 1 & 1 & 1 \\
9 & $\omega_{\loc}$ & 4 & 14 & 0 & 2 & 2 \\
10 & $\omega_{\loc}$ & 4 & 24 & 2 & 2 & 2 \\
11 & $\omega_{\loc}$ & 5 & 5 & 1 & 1 & 1 \\
12 & $\omega_{\loc}$ & 5 & 9 & 1 & 1 & 1 \\
13 & $\omega_{\loc}$ & 5 & 19 & 1 & 2 & 2 \\
14 & $\omega_{\loc}$ & 5 & 34 & 2 & 2 & 2 \\
15 & $\omega_{\loc}$ & 5 & 60 & 2 & 3 & 3 \\
16 & $\omega_{\loc}$ & 6 & 6 & 0 & 2 & 2 \\
17 & $\omega_{\loc}$ & 6 & 11 & 1 & 1 & 1 \\
18 & $\omega_{\loc}$ & 6 & 24 & 0 & 2 & 2 \\
19 & $\omega_{\loc}$ & 6 & 44 & 2 & 2 & 2 \\
20 & $\omega_{\loc}$ & 6 & 80 & 0 & 3 & 3 \\
21 & $\omega_{\loc}$ & 6 & 133 & 3 & 3 & 3 \\
22 & $\omega_{\loc}$ & 7 & 7 & 1 & 1 & 1 \\
23 & $\omega_{\loc}$ & 7 & 13 & 1 & 1 & 1 \\
24 & $\omega_{\loc}$ & 7 & 29 & 1 & 2 & 2 \\
25 & $\omega_{\loc}$ & 7 & 54 & 2 & 2 & 2 \\
26 & $\omega_{\loc}$ & 7 & 100 & 2 & 3 & 3 \\
27 & $\omega_{\loc}$ & 7 & 169 & 3 & 3 & 3 \\
28 & $\omega_{\loc}$ & 7 & 287 & 3 & 4 & 4 \\
29 & $\omega_{\loc}$ & 8 & 8 & 0 & 2 & 2 \\
30 & $\omega_{\loc}$ & 8 & 15 & 1 & 1 & 1 \\
\midrule
1 & $\dim_{\loc}$ & 0 & 1 & 1 & 0 & 0 \\
2 & $\dim_{\loc}$ & 1 & 2 & -- & -- & -- \\
3 & $\dim_{\loc}$ & 1 & 3 & 1 & 1 & 1 \\
4 & $\dim_{\loc}$ & 2 & 3 & 1 & 1 & 1 \\
5 & $\dim_{\loc}$ & 2 & 5 & 1 & 1 & 1 \\
6 & $\dim_{\loc}$ & 2 & 9 & 1 & 2 & 2 \\
7 & $\dim_{\loc}$ & 3 & 4 & 0 & 2 & 2 \\
8 & $\dim_{\loc}$ & 3 & 7 & 1 & 1 & 1 \\
9 & $\dim_{\loc}$ & 3 & 14 & 0 & 2 & 2 \\
10 & $\dim_{\loc}$ & 3 & 24 & 2 & 2 & 2 \\
11 & $\dim_{\loc}$ & 4 & 5 & 1 & 1 & 1 \\
12 & $\dim_{\loc}$ & 4 & 9 & 1 & 1 & 1 \\
13 & $\dim_{\loc}$ & 4 & 19 & 1 & 2 & 2 \\
14 & $\dim_{\loc}$ & 4 & 34 & 2 & 2 & 2 \\
15 & $\dim_{\loc}$ & 4 & 60 & 2 & 3 & 3 \\
16 & $\dim_{\loc}$ & 5 & 6 & 0 & 2 & 2 \\
17 & $\dim_{\loc}$ & 5 & 11 & 1 & 1 & 1 \\
18 & $\dim_{\loc}$ & 5 & 24 & 0 & 2 & 2 \\
19 & $\dim_{\loc}$ & 5 & 44 & 2 & 2 & 2 \\
20 & $\dim_{\loc}$ & 5 & 80 & 0 & 3 & 3 \\
21 & $\dim_{\loc}$ & 5 & 133 & 3 & 3 & 3 \\
22 & $\dim_{\loc}$ & 6 & 7 & 1 & 1 & 1 \\
23 & $\dim_{\loc}$ & 6 & 13 & 1 & 1 & 1 \\
24 & $\dim_{\loc}$ & 6 & 29 & 1 & 2 & 2 \\
25 & $\dim_{\loc}$ & 6 & 54 & 2 & 2 & 2 \\
26 & $\dim_{\loc}$ & 6 & 100 & 2 & 3 & 3 \\
27 & $\dim_{\loc}$ & 6 & 169 & 3 & 3 & 3 \\
28 & $\dim_{\loc}$ & 6 & 287 & 3 & 4 & 4 \\
29 & $\dim_{\loc}$ & 7 & 8 & 0 & 2 & 2 \\
30 & $\dim_{\loc}$ & 7 & 15 & 1 & 1 & 1 \\
\bottomrule
\end{longtable}
\endgroup

\subsection{Axial versus spinal concentration}

The strict bound
\[
\rho_I^{\mathrm{sp}}(n)\le \rho_I^{\ax}(n)\le \rho_I^{\mathrm{sp}}(n)+1
\]
is fully consistent with the computations. In many cases the two radii coincide; in some cases the spinal radius is smaller. Thus the thin spine sometimes gives a sharper localization of extremal structure, but the data do not justify a stronger uniform claim.

\subsection{Representative examples}

Several examples illustrate distinct aspects of axial morphology. Figure~\ref{fig:axial-small} shows six smaller graphs, arranged in pairs, and Figure~\ref{fig:axial-large} shows four larger examples, arranged one per row. In both figures, red nodes lie on the self-conjugate axis, orange nodes lie in the thin spine but not on the axis, blue nodes lie in $C_n^{(1)}\setminus \Spn_n$, and gray nodes lie outside the narrow central region.

\clearpage
\begin{figure}[p]
\centering
\includegraphics[width=0.96\textwidth]{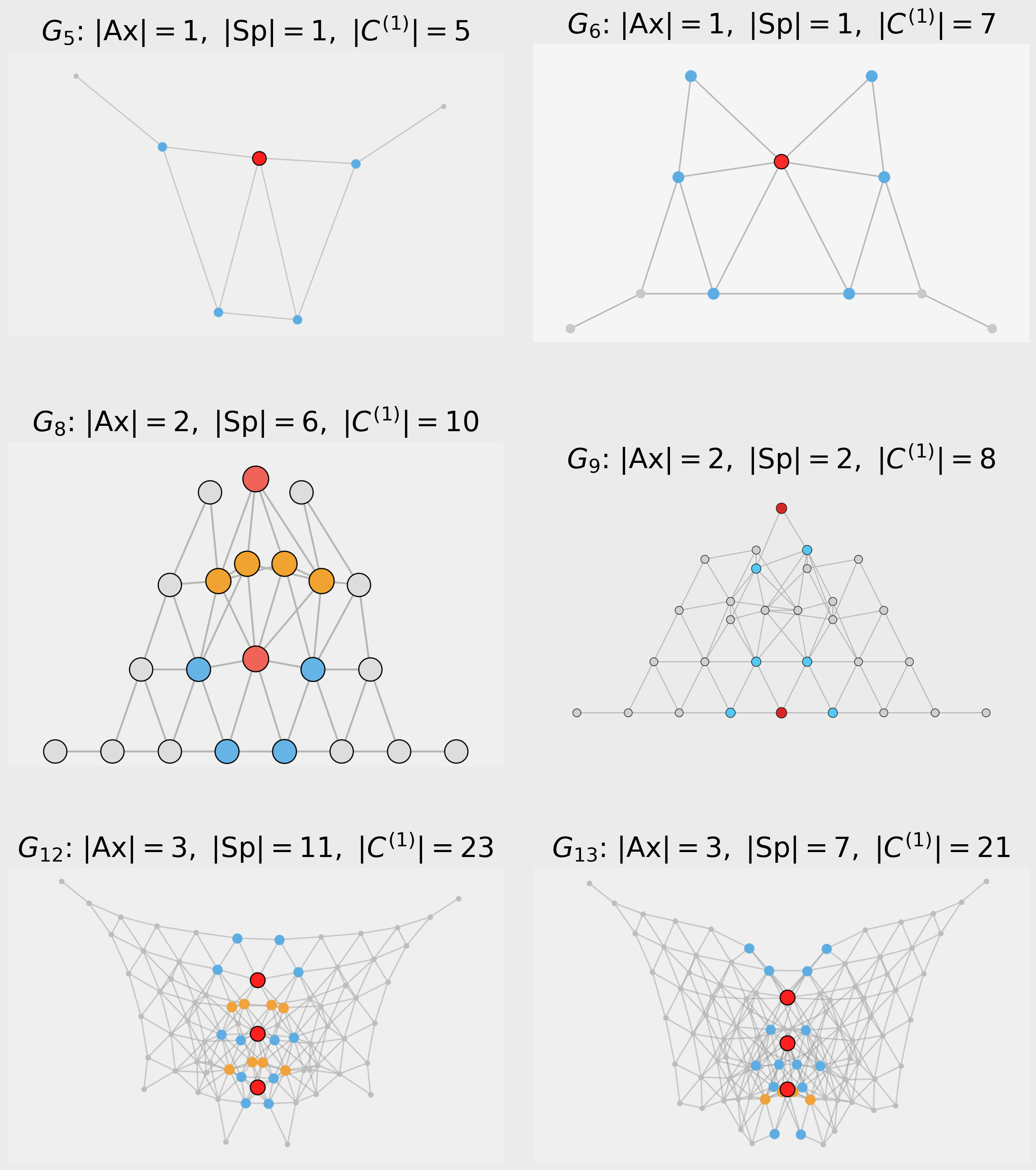}
\caption{Smaller representative examples of axial morphology for $n=5,6,8,9,12,13$.}
\label{fig:axial-small}
\end{figure}

\begin{figure}[p]
\centering
\includegraphics[width=0.72\textwidth]{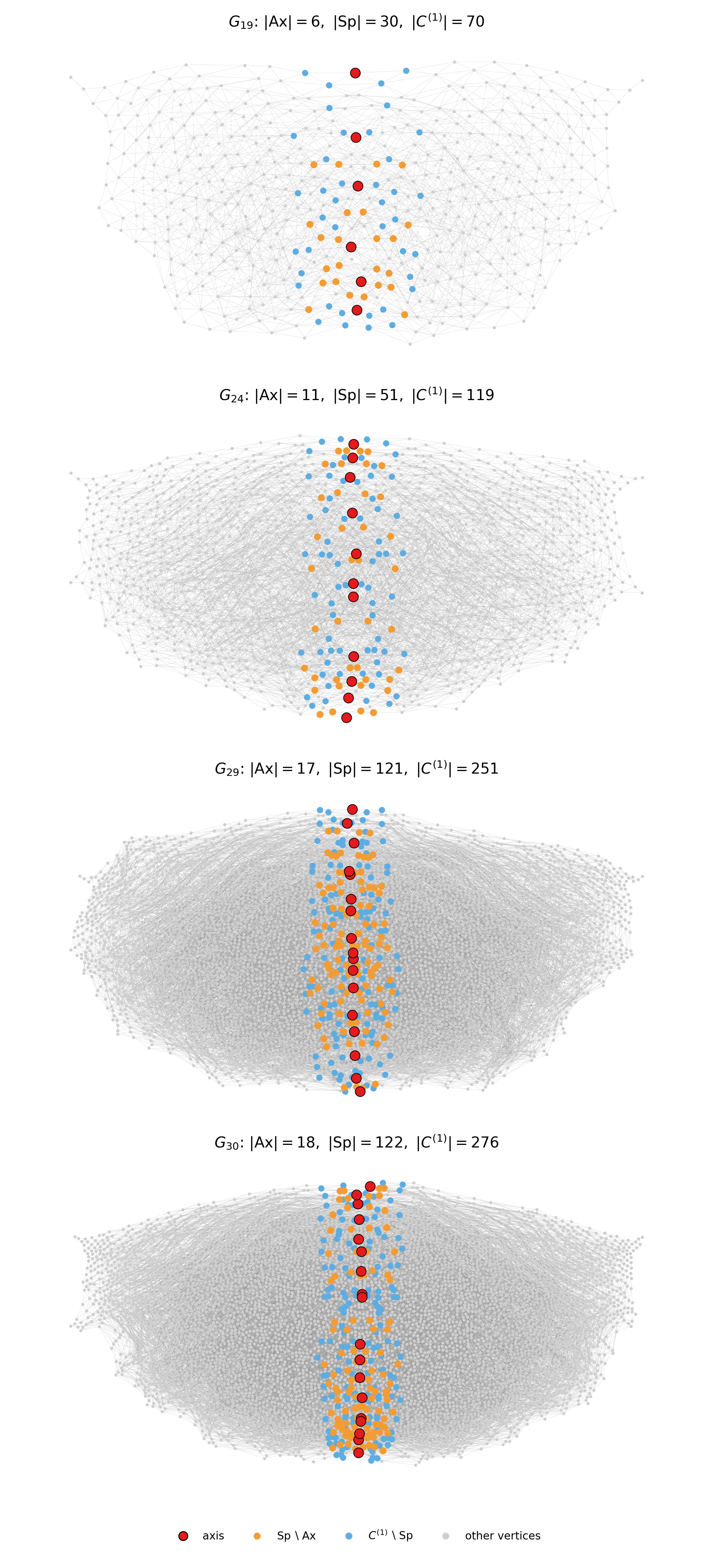}
\caption{Larger representative examples of axial morphology for $n=19,24,29,30$.}
\label{fig:axial-large}
\end{figure}

\paragraph{Example 6.5.1. Compact axial cases: $n=5$ and $n=6$.}
For $n=5$ and $n=6$, the axis consists of a single self-conjugate vertex and the thin spine coincides with the axis:
\[
|\Ax_5|=|\Spn_5|=1,
\qquad
|\Ax_6|=|\Spn_6|=1.
\]
These examples show that the narrow central region may already contain several off-axis vertices even when no axial interaction is present.

\paragraph{Example 6.5.2. The first nontrivial two-vertex axis: $n=8$ and $n=9$.}
For $n=8$, the spine is already substantially larger than the axis:
\[
|\Ax_8|=2,
\qquad
|\Spn_8|=6.
\]
By contrast, for $n=9$ the spine collapses to the axis:
\[
|\Ax_9|=|\Spn_9|=2.
\]
Together these two graphs show that a nontrivial axis does not automatically produce a nontrivial spinal enlargement.

\paragraph{Example 6.5.3. Developed small axial cores: $n=12$ and $n=13$.}
For $n=12$ one has
\[
|\Ax_{12}|=3,
\qquad
|\Spn_{12}|=11,
\qquad
|C_{12}^{(1)}|=23,
\]
and all three local invariants are tightly concentrated near the axis. For $n=13$, the axis still has size $3$, but the spine is smaller,
\[
|\Spn_{13}|=7,
\qquad
|C_{13}^{(1)}|=21,
\]
showing that the central geometry can change noticeably even between consecutive values of $n$.

\paragraph{Example 6.5.4. Larger spinal profiles: $n=19,24,29,30$.}
For $n=19$, the spinal viewpoint improves localization for degree:
\[
\rho_{\deg}^{\ax}(19)=1,
\qquad
\rho_{\deg}^{\mathrm{sp}}(19)=0.
\]
The graphs for $n=24,29,30$ show the same qualitative picture on a larger scale: the axis remains a narrow distinguished set, the thin spine forms a thicker highlighted strip around it, and the first central neighborhood $C_n^{(1)}$ extends beyond the spine. In particular,
\[
|\Ax_{24}|=11,
\ |\Spn_{24}|=51,
\ |C_{24}^{(1)}|=119,
\]
\[
|\Ax_{29}|=17,
\ |\Spn_{29}|=121,
\ |C_{29}^{(1)}|=251,
\]
and
\[
|\Ax_{30}|=18,
\ |\Spn_{30}|=122,
\ |C_{30}^{(1)}|=276.
\]
These examples make visible the distinction between the thin spine and the larger central region.

The highlighted vertices in these figures should therefore be read as follows: the axis records the self-conjugate core itself, the thin spine records those off-axis vertices that simultaneously touch two distinct axial vertices, and the blue region consists of vertices that are still at axial distance $1$ but do not play this mediating role.

\subsection{Axisless case}

In the tested range, the only axisless graph is $G_2$. Its two vertices,
\[
(2),\qquad (1,1),
\]
are conjugate and adjacent, but neither is self-conjugate. Hence $\Ax_2=\varnothing$.

Axisless graphs lie outside the direct scope of the present framework, since axial distance, central regions, and the thin spine are no longer defined in their present form. They are therefore recorded separately.

\section{Conclusion}

We introduced an axial framework for the partition graph $G_n$ based on the self-conjugate axis $\Ax_n$, its distance neighborhoods, and the thin spine $\Spn_n$.

The main strict results are these: the axis is the fixed-point set of conjugation, distinct self-conjugate vertices are never adjacent, the thin spine is a canonical conjugation-invariant induced subgraph, and axial and spinal concentration radii differ by at most one. These statements isolate a definite structural layer of the graph.

The computational results show that, in the tested range $1\le n\le 30$, maximizers of the principal local invariants remain close to the axis and the spine. This concentration is especially sharp for degree and remains visible for the clique-based invariants.

Taken together, these results justify the use of the axis and the spine as natural objects for describing the structure of $G_n$. The paper does not establish an asymptotic theory of concentration, but it provides a precise framework in which such questions can be formulated.

\section{Problems and Outlook}

Several natural questions remain open.

\subsection{Bounded concentration}

\paragraph{Problem 8.1.}
Determine whether, for each invariant considered here, there exists a constant $r$ such that all maximizers lie in $C_n^{(r)}$ for all sufficiently large axial $n$.

\paragraph{Problem 8.2.}
Determine whether there exists a constant $s$ such that all maximizers lie in $\Spn_n^{(s)}$ for all sufficiently large axial $n$.

\subsection{The role of the thin spine}

\paragraph{Problem 8.3.}
Determine whether the thin spine already captures the correct asymptotic locus of extremal local complexity, or whether thicker neighborhoods $\Spn_n^{(r)}$ are required.

\paragraph{Problem 8.4.}
Identify graph-theoretic properties of $\Spn_n$ that would justify treating it as a genuine core, or at least as a plausible skeleton candidate, for $G_n$.

\subsection{The axial interaction graph}

\paragraph{Problem 8.5.}
Study the combinatorial structure of the axial interaction graph $\mathcal A_n$: connectedness, diameter, edge density, and dependence on $n$.

\paragraph{Problem 8.6.}
Relate the structure of $\mathcal A_n$ to the location of high-degree vertices and locally high-dimensional simplicial structure.

\subsection{Radial profiles}

\paragraph{Problem 8.7.}
Study the shell distributions
\[
s_{\ax}(n,k),\qquad s_{\mathrm{sp}}(n,k),
\]
and the cumulative masses
\[
|C_n^{(r)}|,\qquad |\Spn_n^{(r)}|.
\]

\paragraph{Problem 8.8.}
Determine whether these radial profiles exhibit stable normalized behavior as $n$ grows.

\subsection{Axisless graphs}

\paragraph{Problem 8.9.}
Construct a canonical replacement for the self-conjugate axis when $\Ax_n=\varnothing$, and determine whether an analogue of the spinal construction can be defined there.

\subsection{Axial levels and compact nodes}

A natural direction is to refine the self-conjugate axis by introducing an internal order and studying its decomposition into axial levels, compact nodes, and block thresholds. One may ask whether there is a useful level decomposition of $\Ax_n$ compatible with the present graph-theoretic spine and whether distinguished compact nodes control transitions between different spinal regimes.

\paragraph{Problem 8.10.}
Develop a compatible theory of axial levels, compact nodes, and block thresholds along the ordered self-conjugate axis, and determine how these finer axial structures interact with the thin spine and with concentration of local invariants.

\subsection{Broader context}

The present paper isolates one structural layer of the partition graph: the axial layer determined by conjugation and its first off-axis thickening. A natural continuation is to study how this layer interacts with other features of $G_n$, including local simplicial structure, anisotropy, and growth across $n$. For neighboring combinatorial viewpoints on partitions, order structures, and minimal-change generation, compare~\cite{Brylawski1973Lattice,GreeneKleitman1986LongestChains,LatapyPhan2009Lattice,Mutze2023UpdatedSurvey,RasmussenSavageWest1995GrayCodeFamilies,Savage1989GrayCodeSequences}.

\section*{Acknowledgements}
The author acknowledges the use of ChatGPT (OpenAI) for discussion, structural planning, and editorial assistance during the preparation of this manuscript. All mathematical statements, proofs, computations, and final wording were checked and approved by the author, who takes full responsibility for the contents of the paper.

\end{document}